\documentclass[12pt, a4paper, oneside]{amsart}
\usepackage{amsmath, amsthm, amssymb, geometry}
\usepackage[colorlinks=false]{hyperref}

\newtheorem{thm}{Theorem}[section]
\newtheorem{prop}[thm]{Proposition}
\newtheorem{lem}[thm]{Lemma}
\newtheorem{cor}[thm]{Corollary}

\theoremstyle{definition}
\newtheorem{defn}[thm]{Definition}

\theoremstyle{remark}
\newtheorem{ex}[thm]{Example}

\newtheorem{rem}[thm]{Remark}

\def\N{\mathbb{N}}

\def\R{\mathbb{R}}
\def\C{\mathbb{C}}

\def\D{\mathbb{D}}
\def\B{\mathbb{B}}

\def\cO{\mathcal{O}}
\def\cP{\mathcal{P}}

\def\cS{\mathcal{S}}

\def\cU{\mathcal{U}}
\def\cV{\mathcal{V}}
\def\cW{\mathcal{W}}

\def\cZ{\mathcal{Z}}

\def\fI{\mathfrak{I}}

\def\Aut{\operatorname{Aut}}

\def\id{\operatorname{id}}
\def\ev{\operatorname{ev}}

\begin{document}

\title[Dense holomorphic curves in spaces of holomorphic maps]{Dense holomorphic curves in spaces of holomorphic maps and applications to universal maps}
\author{Yuta Kusakabe}
\address{Department of Mathematics, Graduate School of Science, Osaka University, Toyonaka, Osaka 560-0043, Japan}
\email{y-kusakabe@cr.math.sci.osaka-u.ac.jp}

\begin{abstract}
We study when there exists a dense holomorphic curve in a space of holomorphic maps from a Stein space.
We first show that for any bounded convex domain $\Omega\Subset\C^n$ and any connected complex manifold $Y$, the space $\cO(\Omega,Y)$ contains a dense holomorphic disc.
Our second result states that $Y$ is an Oka manifold if and only if for any Stein space $X$ there exists a dense entire curve in every path component of $\cO(X,Y)$.

In the second half of this paper, we apply the above results to the theory of universal functions.
It is proved that for any bounded convex domain $\Omega\Subset\C^n$, any fixed-point-free automorphism of $\Omega$ and any connected complex manifold $Y$, there exists a universal map $\Omega\to Y$.
We also characterize Oka manifolds by the existence of universal maps.
\end{abstract}

\maketitle

\section{Introduction}\label{sec:intro}

Let $Y$ be a complex space (throughout this paper always taken to be reduced and second countable).
It is a fundamental problem whether a given complex space $Z$ admits a holomorphic map $Z\to Y$ with dense image.
In this paper, we study the following more general problem for a Stein space $X$:
When does there exist a dense holomorphic map from $Z$ to a subspace $\cS$ of the space $\cO(X,Y)$ of holomorphic maps from $X$ to $Y$ equipped with the compact-open topology?
Here we say a map $f:Z\to\cS$ to be holomorphic if the associated map $Z\times X\to Y,\ (z,x)\mapsto f(z)(x)$ is holomorphic in the usual sense.
Let us introduce the following terminology.

\begin{defn}
A subspace $\cS\subset\cO(X,Y)$ is said to be {\it $Z$-dominated} if there exists a holomorphic map $Z\to\cS$ with dense image.
\end{defn}

Let us review the case when $X$ is a singleton, i.e. $\cO(X,Y)\cong Y$.
Winkelmann \cite{Win} proved that for any irreducible complex space $Z$ admitting a nonconstant bounded holomorphic function, the unit disc $\D$ is $Z$-dominated.
Correspondingly, Chen and Wang \cite{CheWan} proved that for any irreducible complex space $Z$ admitting a nonconstant holomorphic function, the complex line $\C$ is $Z$-dominated.
Since the composition of two dense holomorphic maps has dense image, our problem reduces to the $\D$-dominability or the $\C$-dominability of a subspace $\cS\subset\cO(X,Y)$ in most cases even if $X$ is not a singleton.
For this reason, we study only the $\D$-dominability and the $\C$-dominability.

In his paper \cite{Win}, Winkelmann also proved that any irreducible complex space is $\D$-dominated.
Our first main result is the following partial generalization to spaces of holomorphic maps.
It is proved in Section \ref{sec:ddom}.

\begin{thm}\label{thm:ddom}
Let $\Omega\Subset\C^n$ be a bounded convex domain and $Y$ a connected complex manifold.
Then $\cO(\Omega,Y)$ is $\D$-dominated.
\end{thm}

If we replace $\Omega$ by a more general pseudoconvex domain or replace $Y$ by a singular space, then this theorem is no longer true (see examples at the end of Section \ref{sec:ddom}).
This is due to the reducibility with respect to the Zariski topologies on the spaces of holomorphic maps (see Definition \ref{defn:zar}).

In Section \ref{sec:cdom}, we study the $\C$-dominability.
A connected Oka manifold (see Definition \ref{defn:oka}) is a typical example of a $\C$-dominated space, but the converse is not true.
For example, for a discrete sequence $D\subset\C^2$ which is unavoidable by nondegenerate holomorphic self-maps of $\C^2$ (cf. \cite[\S4.7]{For1}), the complement $\C^2\setminus D$ is $\C$-dominated for dimensional reasons, but it is not Oka.
By the following second main result, however, a complex manifold $Y$ must be Oka if we further assume that {\it every} space of holomorphic maps to $Y$ is $\C$-dominated.

\begin{thm}\label{thm:cdomoka}
For a connected complex manifold $Y$, the following are equivalent:
\begin{enumerate}
\item $Y$ is Oka.
\item For any Stein space $X$, every path component $\cP\subset\cO(X,Y)$ is $\C$-dominated.
\item For any bounded convex domain $\Omega\Subset\C^n$, the space $\cO(\Omega,Y)$ is $\C$-dominated.
\end{enumerate}
\end{thm}

In Section \ref{sec:univ}, we give applications of the above results to the theory of universal functions.
In 1929, Birkhoff constructed for each $a\in\C\setminus\{0\}$ an entire function $F(z)\in\cO(\C)$ such that $\{F(z+ja)\}_{j\in\N}$ is dense in $\cO(\C)$ (cf. \cite{GroPer}).
Such a function is called a {\it universal function}.
Universal functions $X\to\C$ on more general spaces $X$ have been studied intensively.
In this paper, we make the first attempt to develop the theory of universal maps $X\to Y$ to more general targets $Y$.

\begin{defn}
Let $\tau\in\Aut X$ be an automorphism and assume that $\cS\subset\cO(X,Y)$ is a {\it $\tau^*$-invariant} subspace, i.e. $\tau^*\cS\subset\cS$, where $\tau^*:\cO(X,Y)\to\cO(X,Y)$ is the precomposition map.
A holomorphic map $F\in\cS$ is called an {\it $\cS$-universal map for $\tau$} if $\{F\circ\tau^j\}_{j\in\N}$ is dense in $\cS$.
\end{defn}



More recently, Zaj\c{a}c \cite{Zaj} characterized automorphisms admitting universal functions, or universal maps to $\C$.
Andrist and Wold \cite{AndWol} call such automorphisms {\it generalized translations} (see Definition \ref{defn:gentransl}).
It is natural to consider universal maps for generalized translations.
For a bounded convex domain, a generalized translation is just a fixed-point-free automorphism (see Proposition \ref{prop:fixedpts}).
The following is proved as an application of Theorem \ref{thm:ddom}.

\begin{thm}\label{thm:univcvx}
Let $\Omega\Subset\C^n$ be a bounded convex domain, $\tau\in\Aut\Omega$ a generalized translation and $Y$ a connected complex manifold.
Then there exists an $\cO(\Omega,Y)$-universal map for $\tau$.
\end{thm}

For a universal map from a more general Stein space, we have the following result which is an application of Theorem \ref{thm:cdomoka}.

\begin{thm}\label{thm:univoka}
For a connected complex manifold $Y$, the following are equivalent:
\begin{enumerate}
\item $Y$ is Oka.
\item For any Stein space $X$, any generalized translation $\tau\in\Aut X$ and any $\tau^*$-invariant path component $\cP\subset\cO(X,Y)$, there exists a $\cP$-universal map.
\item For any convex domain $\Omega\subset\C^n$ and any generalized translation $\tau\in\Aut\Omega$, there exists an $\cO(\Omega,Y)$-universal map.
\end{enumerate}
\end{thm}

We also have an application of universal maps.
Forstneri\v{c} and Winkelmann \cite{ForWin} proved that the set of dense holomorphic discs in a connected complex manifold is dense in the space of holomorphic discs.
By Theorem \ref{thm:univcvx}, we may easily obtain the following generalization to several variables because a universal map has dense image and a generic map is universal if there exists some universal map.

\begin{cor}
For any connected complex manifold $Y$, the set of dense holomorphic balls $\B^n\to Y$ is dense in $\cO(\B^n,Y)$.
It holds also for a polydisc $\D^n$.
\end{cor}

Finally, let us give a few remarks and notations which will be used in this paper.

\begin{rem}\label{rem:intro}
Let $X,Y$ and $Z$ be complex spaces.
\\
(1) The space $\cO(X,Y)$ of holomorphic maps is second countable and completely metrizable.
In particular, it is a separable Baire space.
We omit the proof but it follows from the assumption that $X$ and $Y$ are second countable.
\\
(2) We denote by $\cO(X,\cO(Y,Z))$ the space of holomorphic maps $X\to\cO(Y,Z)$.
This space is naturally homeomorphic to $\cO(X\times Y,Z)$ by definition.
\\
(3) For a holomorphic map $f:X\to Y$, we denote by $f^*:\cO(Y,Z)\to\cO(X,Z)$ and $f_*:\cO(Z,X)\to\cO(Z,Y)$ the precomposition map and the postcomposition map, respectively.
We denote by ${\ev}:\cO(X,Y)\times X\to Y$ the evaluation map $(g,x)\mapsto g(x)$.
Note that these maps define the postcomposition maps $\cO(S,\cO(Y,Z))\to\cO(S,\cO(X,Z)),\ \cO(S,\cO(Z,X))\to\cO(S,\cO(Z,Y))$ and $\cO(S,\cO(X,Y))\times\cO(S,X)\to\cO(S,Y)$ between spaces of holomorphic maps for any complex space $S$.
\end{rem}

\section{The $\D$-dominability: Proof of Theorem \ref{thm:ddom}}\label{sec:ddom}

Our proof of Theorem \ref{thm:ddom} is based on the ideas of Forstneri\v{c} and Winkelmann \cite{ForWin}.
We shall need two lemmas.

\begin{lem}\label{lem:approx}
Assume that $\Omega\Subset\C^n$ is a bounded convex domain and $Y$ is a connected complex manifold.
Let $f:\D\to\cO(\Omega,Y)$ be a holomorphic disc and $\cU\subset\cO(\Omega,Y)$ a nonempty open subset.
Then there exists a sequence $\{f_j:W_j\to\cO(\Omega,Y)\}_{j\in\N}$ of holomorphic maps from open neighborhoods of $\overline\D\cup[1,2]\subset\C$ such that
\begin{enumerate}
\item $f_j|_\D\to f$ as $j\to\infty$, and
\item $f_j(2)\in\cU$ for all $j\in\N$.
\end{enumerate}
\end{lem}

To prove this lemma, we need the following special case of Theorem 3.7.2 in \cite{For1}.
Note that for any compact polynomially convex subset $0\in L\subset\C^n$ the subset $((\overline\D\cup\{2\})\times L)\cup([1,2]\times\{0\})\subset\C^{n+1}$ is also polynomially convex.

\begin{prop}\label{prop:preapprox}
Assume that $0\in\Omega\Subset\C^n$ is a bounded convex domain and $Y$ is a complex manifold.
Set $K=(\overline\D\cup\{2\})\times\overline\Omega$ and $I=[1,2]\times\{0\}$.
Let $W\subset\C^{n+1}$ be an open neighborhood of $K$ and $f:W\cup I\to Y$ a continuous map which is holomorphic on $W$.
Then there exists a sequence $\{f_j:W_j\to Y\}_{j\in\N}$ of holomorphic maps from open neighborhoods of $K\cup I\subset\C^{n+1}$ such that
\begin{enumerate}
\item $f_j|_{K\cup I}\to f|_{K\cup I}$ as $j\to\infty$, and
\item $f_j(2,\cdot)|_\Omega=f(2,\cdot)|_\Omega$ for all $j\in\N$.
\end{enumerate}
\end{prop}

For $a\in\C$ and $S\subset \C^n$, we write $aS=\{az:z\in S\}\subset\C^n$.

\begin{proof}[Proof of Lemma \ref{lem:approx}]
We may assume that $\Omega$ contains the origin $0\in\C^n$.
Let us consider the associated holomorphic map $\widehat f:\D\times\Omega\to Y,\ (z,x)\mapsto f(z)(x)$.
For each $r>1$, the map $(z,x)\mapsto\widehat f(r^{-1}z,r^{-1}x)$ is holomorphic on the neighborhood $r(\D\times\Omega)$ of $\overline\D\times\overline\Omega$ and converges to $\widehat f$ on $\D\times\Omega$ as $r\searrow1$.
Therefore, we may assume that $\widehat f$ is holomorphic on a neighborhood of $\overline\D\times\overline\Omega$ from the beginning.
Take a holomorphic map $u\in\cU$.
For the same reason, we may assume that $u$ is holomorphic on a neighborhood of $\overline\Omega$.

Since $Y$ is connected, we may easily find an open neighborhood $V\subset\C^{n+1}$ of $(\overline\D\cup\{2\})\times\overline\Omega$ and a continuous map $g:V\cup([1,2]\times\{0\})\to Y$ such that $g|_V$ is holomorphic, $g|_{\D\times\Omega}=\widehat f$ and $g(2,\cdot)|_\Omega=u$.
By Proposition \ref{prop:preapprox}, there exists a sequence $\{g_j:V_j\to Y\}_{j\in\N}$ of holomorphic maps from open neighborhoods of $((\overline\D\cup\{2\})\times\overline\Omega)\cup([1,2]\times\{0\})\subset\C^{n+1}$ such that
\begin{enumerate}
\renewcommand{\labelenumi}{(\alph{enumi})}
\item $g_j|_{\D\times\Omega}\to\widehat f$ as $j\to\infty$, and
\item $g_j(2,\cdot)|_\Omega=u$ for all $j\in\N$.
\end{enumerate}

For each fixed $j\in\N$, choose open sets $\overline\D\cup\{2\}\subset V_{\overline\D\cup\{2\}}\subset\C,\ \overline\Omega\subset V_{\overline\Omega}\subset\C^n,\ [1,2]\subset V_{[1,2]}\subset\C$ and $0\in V_0\subset\C^n$ such that $(V_{\overline\D\cup\{2\}}\times V_{\overline\Omega})\cup(V_{[1,2]}\times V_0)\subset V_j.$
Since $t\overline\Omega\subset V_{\overline\Omega}$ for all $t\in[0,1]$, there exists an open neighborhood $N\subset\C$ of $[0,1]$ such that $z\Omega\subset V_{\overline\Omega}$ for all $z\in N$.
Choose a small number $\varepsilon>0$ such that $z\Omega\subset V_0$ for all $z\in\varepsilon\D$.
Note that there exists a continuous function $\chi$ on a neighborhood of $\overline\D\cup[1,2]$ with values in $[0,1]\subset N$ such that $\chi\equiv1$ on a neighborhood of $\overline\D\cup\{2\}$ and $\chi\equiv0$ on $[1,2]\setminus V_{\overline\D\cup\{2\}}$.
Again by Proposition \ref{prop:preapprox}, there exists a holomorphic map $\rho_j:W_j\to N$ from an open neighborhood $W_j\subset V_{\overline\D\cup\{2\}}\cup V_{[1,2]}$ of $\overline{\D}\cup[1,2]$ such that
\begin{enumerate}
\renewcommand{\labelenumi}{(\roman{enumi})}
\item $\rho_j|_\D$ is (arbitrarily) close to the constant map $1$,
\item $|\rho_j(z)|<\varepsilon$ for all $z\in[1,2]\setminus V_{\overline\D\cup\{2\}}$, and
\item $\rho_j(2)=1$.
\end{enumerate}
Shrinking $W_j$ if necessary, we may assume that $|\rho_j(z)|<\varepsilon$ for all $z\in W_j\setminus V_{\overline\D\cup\{2\}}$.
By our construction, the map $(z,x)\mapsto g_j(z,\rho_j(z)x)$ is defined and holomorphic on $W_j\times\Omega$.
We denote this map by $\widehat f_j:W_j\times\Omega\to Y$.

By the properties (a) and (i), choosing $\{\rho_j\}_{j\in\N}$ suitably, we may assume that $\widehat f_j|_{\D\times\Omega}\to\widehat f|_{\D\times\Omega}$ as $j\to\infty$.
The properties (b) and (iii) implies that $\widehat f_j(2,\cdot)|_\Omega=u\in\cU$ for all $j\in\N$.
Therefore, the associated maps $f_j:W_j\to\cO(\Omega,Y)$ to $\widehat f_j$ have the required properties.
\end{proof}

The following lemma is an immediate consequence of Lemma 1 in \cite{ForWin}.

\begin{lem}\label{lem:adjust}
Let $\{W_j\}_{j\in\N}$ be a sequence of open neighborhoods of $\overline\D\cup[1,2]\subset\C$.
Then there exists a sequence $\{\varphi_j:\D\to W_j\}_{j\in\N}$ of holomorphic discs such that
\begin{enumerate}
\item $\varphi_j\to\id_\D$ as $j\to\infty$, and
\item $2\in\varphi_j(\D)$ for all $j\in\N$.
\end{enumerate}
\end{lem}

\begin{proof}[Proof of Theorem \ref{thm:ddom}]
Choose a countable base $\{\cU_j\}_{j\in\N}$ for the topology of $\cO(\Omega,Y)$ and assume that $\cU_j\neq\emptyset$ for all $j\in\N$.
Let us consider the open subsets $\cV_j=\{f\in\cO(\D,\cO(\Omega,Y)):f(\D)\cap\cU_j\neq\emptyset\}$ of $\cO(\D,\cO(\Omega,Y))$.
Note that $\bigcap_{j\in\N}\cV_j$ is the set of dense holomorphic discs $\D\to\cO(\Omega,Y)$.
It suffices to show that each $\cV_j$ is dense in $\cO(\D,\cO(\Omega,Y))$ because $\cO(\D,\cO(\Omega,Y))$ is a Baire space (see Remark \ref{rem:intro}).

For any holomorphic disc $f:\D\to\cO(\Omega,Y)$, by Lemma \ref{lem:approx}, there exists a sequence $\{f_k:W_k\to\cO(\Omega,Y)\}_{k\in\N}$ of holomorphic maps from open neighborhoods of $\overline\D\cup[1,2]$ such that $f_k|_\D\to f$ as $k\to\infty$ and $f_k(2)\in\cU_j$ for all $k\in\N$.
By Lemma \ref{lem:adjust}, there exists a sequence $\{\varphi_k:\D\to W_k\}_{k\in\N}$ of holomorphic discs such that $f_k\circ\varphi_k\to f$ as $k\to\infty$ and $f_k\circ\varphi_k(\D)\cap\cU_j\neq\emptyset$ for all $k\in\N$.
This proves that $\cV_j$ is dense in $\cO(\D,\cO(\Omega,Y))$.
\end{proof}

\begin{rem}\label{rem:densedense}
The above proof also shows that for a connected complex manifold $Y$ the set of dense holomorphic discs $\D\to\cO(\Omega,Y)$ is dense in $\cO(\D,\cO(\Omega,Y))$.
In fact, if $\cO(\D\times\Omega,Y)$ is $\D$-dominated for a complex space $Y$, the set of dense holomorphic discs $\D\to\cO(\Omega,Y)$ must be dense.
This can be seen as follows.
By Corollary \ref{cor:ddomuniv}, the $\D$-dominability of $\cO(\D\times\Omega,Y)$ implies that there exists an $\cO(\D\times\Omega,Y)$-universal map for a generalized translation of the form $\tau=\sigma\times\id_\Omega$.
This gives an $\cO(\D,\cO(\Omega,Y))$-universal map for $\sigma$ (in the obvious sense).
Then it follows that a generic holomorphic disc $\D\to\cO(\Omega,Y)$ is an $\cO(\D,\cO(\Omega,Y))$-universal  map and has dense image.
\end{rem}



In fact, we may obtain the following more precise result by the arguments in \cite{ForWin} and the above lemmas with minor modifications, but we omit its proof.
We denote by $\cO(\overline{\Omega},Y)|_{\Omega}$ the set of holomorphic maps $\Omega\to Y$ which can be extended to open neighborhoods of $\overline\Omega$.

\begin{prop}\label{prop:ddom}
Let $\Omega\Subset\C^n$ be a bounded convex domain and $Y$ a connected complex manifold.
Then for every countable subset $\cS\subset\cO(\overline{\Omega},Y)|_{\Omega}$, the set $\{f\in\cO(\D,\cO(\Omega,Y)):\cS\subset f(\D)\}$ is dense in $\cO(\D,\cO(\Omega,Y))$.
\end{prop}

Note that we cannot omit the assumption $\cS\subset\cO(\overline{\Omega},Y)|_{\Omega}$ because bounded convex domains are immobile.
Recall that a complex space $Y$ is said to be {\it immobile} if every holomorphic map $f:\D\times Y\to Y$ such that $f(0,y)=y$ for all $y\in Y$ satisfies $f(z,y)=y$ for all $(z,y)\in\D\times Y$.
It is equivalent to saying that every nonconstant holomorphic disc $\D\to\cO(Y,Y)$ does not intersect $\Aut Y$.
It is known that every Kobayashi hyperbolic space is immobile (cf. \cite[Theorem 5.4.5]{Kob}).

We end this section with some examples of spaces of holomorphic maps which are not $\D$-dominated.
Let us introduce the notion of a Zariski topology on a space of holomorphic maps because these examples can be explained by the reducibility with respect to their Zariski topologies.
This notion is also used in Section \ref{sec:univ}.

\begin{defn}\label{defn:zar}
We say a closed subset $\cS\subset\cO(X,Y)$ to be {\it Zariski closed} if for any holomorphic map $f:Z\to\cO(X,Y)$ from any complex space $Z$, the preimage $f^{-1}(\cS)\subset Z$ is Zariski closed.
Zariski closed subsets define the {\it Zariski topology} on $\cO(X,Y)$.
\end{defn}

Note that a holomorphic map $f:Z\to\cO(X,Y)$ is continuous with respect to the Zariski topology.
Thus, if $Z$ is irreducible, the closure $\overline{f(Z)}$ must be contained in some irreducible component of $\cO(X,Y)$.
In particular, if $\cO(X,Y)$ is reducible, it is not $\D$-dominated.

\begin{ex}
Let $Y$ be a complex space containing a nontrivial entire curve $\C\to Y$.
Assume that there exists a proper closed complex subvariety $Y'\subsetneq Y$ which contains all nontrivial entire curves.
Then $\cO(\C,Y)$ is reducible.
Indeed, if we identify the set of constant maps $\C\to Y$ with $Y$, then $Y$ and $\cO(\C,Y')$ are proper Zariski closed subset of $\cO(\C,Y)$ and we have $\cO(\C,Y)=Y\cup\cO(\C,Y')$.
Thus, we cannot omit the assumption that $\Omega$ is bounded in Theorem \ref{thm:ddom}.
\end{ex}

\begin{ex}
Let $A=\{z\in\C:1/r<|z|<r\},\ r>1$ be an annulus.
Note that $\cO(A,A)$ has three path components, namely the component $\cP_0$ of null-homotopic maps, the component $\cP_1$ of degree one maps and the component $\cP_{-1}$ of degree minus one maps.
Then $\cP_0$ is $\D$-dominated (it follows from the $\D$-dominability of $\cO(A,\D)$), but others are not $\D$-dominated because $A$ is immobile and $\Aut A=\cP_1\cup\cP_{-1}$.
It also follows that the Zariski topology on $\Aut A$ coincides with the compact-open topology.
Thus $\cP_1$ and $\cP_{-1}$ are highly reducible.
By this example, we cannot replace the assumption that $\Omega$ is convex by the pseudoconvexity in Theorem \ref{thm:ddom} even if we consider the $\D$-dominability of path components.
\end{ex}

\begin{ex}
Forstneri\v{c} and Winkelmann \cite{ForWin} showed that there exists an irreducible singular surface $Y$ such that the set of dense holomorphic discs $\D\to Y$ is not dense in $\cO(\D,Y)$.
By Remark \ref{rem:densedense}, $\cO(\D,Y)$ is not $\D$-dominated for such a surface $Y$.
Thus, we cannot omit the assumption that $Y$ is nonsingular in Theorem \ref{thm:ddom}.
\end{ex}



\section{The $\C$-dominability and Oka manifolds: Proof of Theorem \ref{thm:cdomoka}}\label{sec:cdom}

We briefly review the definition of an Oka manifold and Forstneri\v{c}'s Oka principle.
For further details on Oka theory, we refer to the comprehensive monograph \cite{For1} and the survey \cite{For2}.

Let $z=(z_1,\ldots,z_n)$, with $z_j=x_j+iy_j$, denote the complex coordinates on $\C^n$.
A map from a compact set $K$ is said to be holomorphic if it is holomorphic on an open neighborhood of $K$.
We denote by $\cO(K,Y)$ the space of holomorphic maps $K\to Y$ equipped with the compact-open topology.

\begin{defn}\label{defn:oka}
A {\it special convex pair} $(K,Q)$ in $\C^n$ consists of a closed cube $Q=\{z\in\C^n:|x_j|\leq a_j,\ |y_j|\leq b_j,\ j=1,\ldots,n\}$ and a compact convex set $K=\{z\in Q:y_n\leq h(z_1,\ldots,z_{n-1},x_n)\}$, where $h$ is a smooth concave function with values in $(-b_n,b_n)$.

A complex manifold $Y$ is called an {\it Oka manifold} if for each special convex pair $(K,Q)$ the restriction map $\cO(Q,Y)\to\cO(K,Y)$ has dense image.
\end{defn}

The most important property of Oka manifolds is that maps from Stein spaces to Oka manifolds enjoy the following Oka principle.

\begin{thm}[{cf. \cite[Theorem 5.4.4]{For1}}]\label{thm:okaprinciple}
Assume that $X$ is a Stein space and $Y$ is an Oka manifold.
Let $K$ be a compact $\cO(X)$-convex subset of $X$, $X'$ a closed complex subvariety of $X$, $P_0\subset P$ compact sets in a Euclidean space $\R^n$ and $f:P\times X\to Y$ a continuous map such that
\begin{enumerate}
\renewcommand{\labelenumi}{(\alph{enumi})}
\item for any $p\in P$, $f(p,\cdot):X\to Y$ is holomorphic on a neighborhood of K (independent of $p$) and on $X'$, and
\item $f(p,\cdot)$ is holomorphic on X for all $p\in P_0$.
\end{enumerate}
Then there exists a homotopy $f_t:P\times X\to Y,\ t\in[0,1]$ with $f_0=f$ such that each $f_t$
satisfies properties $\rm (a)$ and $\rm (b)$, and the following hold:
\begin{enumerate}
\renewcommand{\labelenumi}{(\roman{enumi})}
\item $f_1(p,\cdot)$ is holomorphic on X for all $p\in P$,
\item $f_t$ is uniformly close to $f$ on $P\times K$ for all $t\in[0,1]$, and
\item $f_t=f$ on $(P_0\times X)\cup(P\times X')$ for all $t\in[0,1]$.
\end{enumerate}
\end{thm}

Let us prove Theorem \ref{thm:cdomoka}.
In fact, we shall prove the following stronger result.

\begin{thm}\label{thm:acconn}
Let $Y$ be a connected complex manifold.
Assume that for any bounded convex domain $\Omega\Subset\C^n$ and any pair of nonempty open subsets $\cU,\cV\subset\cO(\Omega,Y)$, there exists an entire curve $f:\C\to\cO(\Omega,Y)$ such that $f(0)\in\cU$ and $f(1)\in\cV$.
Then $Y$ is Oka.
\end{thm}

\begin{proof}
Let $(K,Q)$ be a special convex pair in $\C^n$ and write $K=\{z\in Q:y_n\leq h(z_1,\ldots,z_{n-1},x_n)\}$ as in Definition \ref{defn:oka}.
We would like to show that the restriction map $\cO(Q,Y)\to\cO(K,Y)$ has dense image.

Take an arbitrary open neighborhood $W$ of $K$.
By considering tangent spaces of the graph of the smooth concave function $h$, we may find $\R$-affine functions $h_1,\ldots,h_k$ on $\C^{n-1}\times\R\cong\R^{2n-1}$ such that
\begin{align*}
K\subset\{z\in Q:y_n\leq h_j(z_1,\ldots,z_{n-1},x_n),\ j=1,\ldots,k\}\subset W.
\end{align*}
For each $l=1,\ldots,k$, let $K_l=\{z\in Q:y_n\leq h_j(z_1,\ldots,z_{n-1},x_n),\ j=1,\ldots,l\}$ and $K_0=Q$.
In order to approximate a holomorphic map in $\cO(K,Y)$ which is holomorphic on $W$, it suffices to show that each restriction map $\cO(K_{l-1},Y)\to\cO(K_l,Y)$ has dense image for $l=1,\ldots,k$.
For this purpose, we use a gluing method in Oka theory (cf. \cite[\S5.9]{For1}).

Let $g_0:\Omega\to Y$ be a holomorphic map from an open neighborhood of $K_l$.
We may assume that $\Omega$ is a bounded convex domain.
Let us construct a smooth strongly pseudoconvex Cartan pair $(D_0,D_1)$ (for the notion of a Cartan pair, see \cite[Definition 5.7.1]{For1}) such that if we set $L=\overline{D_0\cap D_1}$ then
\begin{enumerate}
\renewcommand{\labelenumi}{(\alph{enumi})}
\item $K_l\subset D_0\Subset\Omega,\ K_{l-1}\subset D_0\cup D_1$,
\item $K_l\cap\widehat L_{\cO(\C^n)}=\emptyset$, and $K_l\cup\widehat L_{\cO(\C^n)}$ is polynomially convex.
\end{enumerate}
Choose a small number $\varepsilon>0$ such that $A=\{z\in K_{l-1}:y_n\leq h_l(z_1,\ldots,z_{n-1},x_n)+3\varepsilon\}\subset\Omega$ and let $B=\{z\in K_{l-1}:y_n\geq h_l(z_1,\ldots,z_{n-1},x_n)+2\varepsilon\}$.
Then $(A,B)$ is a Cartan pair in $\C^n$ because $A,\ B,\ A\cup B$ and $A\cap B$ are convex and $\overline{A\setminus B}\cap\overline{B\setminus A}=\emptyset$.
Note that $K_l\subset A$ and $K_{l-1}=A\cup B$.
For the open neighborhood $\Omega$ of $A$ and the open neighborhood $U=\{z\in\C^n:y_n>h_k(z_1,\ldots,z_{n-1},x_n)+\varepsilon\}$ of $B$, there exists a smooth strongly pseudoconvex Cartan pair $(D_0,D_1)$ such that $A\subset D_0\Subset\Omega$ and $B\subset D_1\Subset U$ (see \cite[Proposition 5.7.3]{For1}).
Then $(D_0,D_1)$ satisfies the required properties (see Proposition \ref{prop:disjoint}).

Put $m=\dim Y$.
The composition of the projection map $\bar D_1\times\B^m\to\B^m$ and an open embedding map $\B^m\to Y$ gives a dominating local holomorphic spray $\pi:\bar D_1\times\B^m\to Y$ (for the notion of a local holomorphic spray, see \cite[Definition 5.9.1]{For1}).

By assumption there exists an entire curve $f:\C\to\cO(\Omega\times\B^m,Y)$ such that $f(0)(\cdot,0)$ and $f(1)$ approximate sufficiently $g_0$ on $K_l$ and $\pi$ on $L\times\B^m$, respectively.
By the conditions (a) and (b) of $(D_0,D_1)$ and by Oka-Weil theorem, there exists $\varphi\in\cO(\C^{n+m})$ such that $\varphi|_{K_l\times\B^m}$ is sufficiently close to $0$ and $\varphi|_{L\times\B^m}$ is sufficiently close to $1$.
Consider $g={\ev}(f\circ\varphi,\id_{\bar{D}_0\times\B^m}):\bar{D}_0\times\B^m\to Y$.
It is a local holomorphic spray which is sufficiently close to $\pi$ on $L\times\B^m$ and $g(\cdot,0)$ is sufficiently close to $g_0$ on $K_l$.
The closeness between $g$ and $\pi$ on $L\times\B^m$ implies that there exist a small number $0<\delta<1$ and a continuous map $\gamma:L\times\delta\B^m\to L\times\B^m$ of the form $\gamma(z,w)=(z,\psi(z,w))$ which is holomorphic on $(D_0\cap D_1)\times\delta\B^m$, sufficiently close to the inclusion map $L\times\delta\B^m\hookrightarrow L\times\B^m$ and satisfies $g\circ\gamma=\pi$ on $L\times\delta\B^m$ (see \cite[Lemma 5.9.3]{For1}).
Then the closeness between $\gamma$ and the inclusion map implies that there exist holomorphic maps
\begin{align*}
\alpha:D_0\times\delta^2\B^m\to D_0\times\delta\B^m,\ \beta:D_1\times\delta^2\B^m\to D_1\times\delta\B^m
\end{align*}
which are sufficiently close to the inclusion maps and satisfy $\gamma\circ\beta=\alpha$ on $(D_0\cap D_1)\times\delta^2\B^m$ (see \cite[Proposition 5.8.1]{For1}).
Since $g\circ\alpha=\pi\circ\beta$ on $(D_0\cap D_1)\times\delta^2\B^m$, $g\circ\alpha$ and $\pi\circ\beta$ amalgamate into a holomophic map $g':(D_0\cup D_1)\times\delta^2\B^m\to Y$ such that the restricted map $g'(\cdot,0):K_{l-1}\to Y$ approximates $g_0$ on $K_l$.
\end{proof}

\begin{proof}[Proof of Theorem \ref{thm:cdomoka}]
$(1)\implies(2):$ It is an easy application of Oka principle (Theorem \ref{thm:okaprinciple}).
We may choose a dense sequence $\{f_j\}_{j\in\N}\subset\cP$ since $\cP$ is separable.
Since there exists a continuous map $f:\C\times X\to Y$ such that $f(j,\cdot)=f_j$ for all $j\in\N$, there exists a holomorphic map $\widetilde f:\C\times X\to Y$ satisfying the same condition by Theorem \ref{thm:okaprinciple}.
The associated map to $\widetilde f$ gives a dense entire curve $\C\to\cP$.
\\
$(2)\implies(3):$ It is obvious since bounded convex domains are Stein.
\\
$(3)\implies(1):$ Since a dense entire curve in $\cO(\Omega,Y)$ intersects all nonempty open subsets in $\cO(\Omega,Y)$, Theorem \ref{thm:acconn} implies that $Y$ is Oka.
\end{proof}

We explain another characterization of Oka manifolds which can be obtained by Theorem \ref{thm:acconn}.
Let $X$ and $Y$ be complex spaces.
We say a subspace $\cS\subset\cO(X,Y)$ to be {\it strongly $\C$-connected} if any pair of points in $\cS$ can be joined by an entire curve $\C\to\cS$.
Note that for any Stein space $X$ and any Oka manifold $Y$, every path component of $\cO(X,Y)$ is strongly $\C$-connected by Theorem \ref{thm:okaprinciple}.
The following characterization of Oka manifolds is also an immediate corollary of Theorem \ref{thm:acconn}.

\begin{cor}\label{cor:cconn}
For a connected complex manifold $Y$, the following are equivalent:
\begin{enumerate}
\item $Y$ is Oka.
\item For any Stein space $X$, every path component $\cP\subset\cO(X,Y)$ is strongly $\C$-connected.
\item For any bounded convex domain $\Omega\Subset\C^n$, the space $\cO(\Omega,Y)$ is strongly $\C$-connected.
\end{enumerate}
\end{cor}

The above characterization can be reformulated into more interesting form.
For $x\in X$, we write $\ev_x={\ev}(\cdot,x):\cO(X,Y)\to Y$.
As we mentioned in Remark \ref{rem:intro}, for any holomorphic map $f:Z\to\cO(X,Y)$ from a complex space, the composition ${\ev_x}\circ f:Z\to Y$ is holomorphic.
Therefore it is interesting to ask when $(\ev_0,\ev_1):\cO(\C,Y)\to Y^2$ is an Oka map, namely we call it an {\it Oka map} if it is a Serre fibration and it enjoys the parametric Oka property for Euclidean parameter spaces (in the obvious sense, see \cite[Definition 6.14.7]{For1}).

\begin{cor}
For a connected complex manifold $Y$, the following are equivalent:
\begin{enumerate}
\item $Y$ is Oka.
\item $(\ev_0,\ev_1):\cO(\C,Y)\to Y^2$ is an Oka map.
\item For any bounded convex domain $\Omega\Subset\C^n$, every holomorphic map $\Omega\to Y^2$ lifts to a holomorphic map $\Omega\to\cO(\C,Y)$.
\end{enumerate}
\end{cor}

\begin{proof}
The condition (3) is a simple rephrasing of the condition (3) in Corollary \ref{cor:cconn}.
We only need to prove that $(1)\implies(2)$.
Assume that $Y$ is Oka and denote by $\iota:\{0,1\}\hookrightarrow\C$ the inclusion map.
Note that $(\ev_0,\ev_1):\cO(\C,Y)\to Y^2$ can be identified with the precomposition map $\iota^*:\cO(\C,Y)\to\cO(\{0,1\},Y)$.
Therefore $(\ev_0,\ev_1)$ is a Serre fibration (see \cite[\S16]{Lar}).
By considering the associated map to $P\times X\to\cO(\C,Y)$ and using the parametric Oka property of $Y$ (see Theorem \ref{thm:okaprinciple}), we may easily obtain the parametric Oka property of $(\ev_0,\ev_1)$.
\end{proof}

\section{Universal maps: Applications of dense holomorphic curves}\label{sec:univ}

The following characterizes automorphisms admitting universal functions.

\begin{thm}[cf. {\cite[Theorem 6]{Zaj}}]\label{thm:Zaj}
Let $X$ be a Stein space and $\tau\in\Aut X$ an automorphism.
Then there exists an $\cO(X)$-universal function for $\tau$ if and only if for any compact $\cO(X)$-convex subset $K\subset X$ there exists $j\in\N$ such that $\tau^j(K)\cap K=\emptyset$ and $\tau^j(K)\cup K$ is $\cO(X)$-convex.
\end{thm}

For a connected Stein manifold, the above theorem is due to Zaj\c{a}c \cite{Zaj}.
For arbitrary Stein space, this follows immediately from the following generalization of Lemma 2 in \cite{Zaj} to a Stein space (see the proof of Theorem 6 in \cite{Zaj}).
We may prove the following by choosing a relatively compact open Runge neighborhood of $\widehat K_{\cO(X)}\cup\widehat L_{\cO(X)}$ and embedding it to a Euclidean space, but we omit the details (see the proof of Lemma 2 in \cite{Zaj}).

\begin{prop}[{cf. \cite[Lemma 2]{Zaj}}]\label{prop:disjoint}
Let $X$ be a Stein space.
For compact subsets $K,L\subset X$, the following are equivalent:
\begin{enumerate}
\item There exists $f\in\cO(X)$ such that $\widehat{f(K)}_{\cO(\C)}\cap\widehat{f(L)}_{\cO(\C)}=\emptyset$.
\item $\widehat K_{\cO(X)}\cap\widehat L_{\cO(X)}=\emptyset$ and $\widehat{(K\cup L)}_{\cO(X)}=\widehat K_{\cO(X)}\cup\widehat L_{\cO(X)}$.
\end{enumerate}
In particular, if $K$ and $L$ are disjoint and $\cO(X)$-convex, then $K\cup L$ is $\cO(X)$-convex if and only if $K$ and $L$ satisfy the condition $(1)$.
\end{prop}

Following Theorem \ref{thm:Zaj}, we give the definition of a generalized translation.

\begin{defn}\label{defn:gentransl}
Let $X$ be a Stein space.
We call $\tau\in\Aut X$ a {\it generalized translation} if for any compact $\cO(X)$-convex subset $K\subset X$ there exists $j\in\N$ such that
\begin{enumerate}
\item $\tau^j(K)\cap K=\emptyset$, and
\item $\tau^j(K)\cup K$ is $\cO(X)$-convex.
\end{enumerate}
\end{defn}

The following proposition gives a lot of examples of generalized translations.

\begin{prop}\label{prop:gentransl}
Let $S$ and $X$ be Stein spaces and $f:S\to X$ a holomorphic map.
Assume that $\tau\in\Aut X$ is a generalized translation and $\sigma\in\Aut S$ is an automorphism such that $\tau\circ f=f\circ\sigma$.
Then $\sigma$ is also a generalized translation.
\end{prop}

\begin{proof}
Let $K\subset S$ be a compact $\cO(S)$-convex subset and let $L=\widehat{f(K)}_{\cO(X)}$.
Since $\tau$ is a generalized translation, there exists $j\in\N$ such that $\tau^j(L)\cap L=\emptyset$ and $\tau^j(L)\cup L$ is $\cO(X)$-convex.
By assumption, it follows that $f(\sigma^j(K))\cap f(K)=\emptyset$.
Thus $\sigma^j(K)\cap K=\emptyset$.
The $\cO(S)$-convexity of $\sigma^j(K)\cup K$ follows easily from Proposition \ref{prop:disjoint}.
\end{proof}

We may obtain the following corollary by considering the projection map and the inclusion map.

\begin{cor}\label{cor:gentransl}
Let $X$ be a Stein space and $\tau\in\Aut X$ a generalized translation.
\begin{enumerate}
\item If $Y$ is a Stein space and $\sigma\in\Aut (X\times Y)$ is of the form $\sigma(x,y)=(\tau(x),f(x,y))$, then $\sigma$ is a generalized translation.
\item If $\Omega\subset X$ is a Stein subset such that $\tau(\Omega)=\Omega$, then $\tau|_\Omega\in\Aut\Omega$ is a generalized translation.
\end{enumerate}
\end{cor}

\begin{ex}\label{ex:gentransl}
A usual translation of $\C^n$ is clearly a generalized translation.
By Proposition 4.7, a ball $\B^n$ and a polydisc $\D^n$ have generalized translations (it also follows from Corollary \ref{cor:gentransl} and the Cayley transform $\B^n\cong\{(z',z_n)\in\C^n:\fI z_n>\|z'\|^2\}$).
Therefore, for any Stein space $X$, the product spaces $\C\times X$ and $\D\times X$ have generalized translations of the form $\tau=\sigma\times\id_X$ by Corollary \ref{cor:gentransl}.
\end{ex}

\subsection{Universal maps from bounded convex domains}

The following gives a simple characterization of generalized translations on bounded convex domains.
Recall that $\{\varphi_j\}_{j\in\N}\subset\Aut\Omega$ is said to be {\it compactly divergent} if for any compact subset $K\subset X$ there exists $N\in\N$ such that $\varphi_j(K)\cap K=\emptyset$ for all $j\geq N$.

\begin{prop}\label{prop:fixedpts}
Let $\Omega\Subset\C^n$ be a bounded convex domain.
For an automorphism $\tau\in\Aut\Omega$, the following are equivalent:
\begin{enumerate}
\item $\tau$ is a generalized translation.
\item $\tau$ has no fixed points.
\item $\{\tau^j\}_{j\in\N}$ is compactly divergent.
\end{enumerate}
\end{prop}

\begin{proof}
It is obvious that $(1)\implies(2)$.
Since $\Omega$ is a bounded convex domain, $(2)\iff(3)$ (cf. \cite[Theorem 2.4.20]{Aba}).
The implication $(3)\implies(1)$ follows from Theorem \ref{thm:ddom} and the proof of Corollary \ref{cor:ddomuniv} below (see also \cite[Theorem 11]{Zaj}).
\end{proof}

To prove Theorem \ref{thm:univcvx}, we shall need the following lemma.
A complex space $X$ is said to be {\it $c$-finitely compact} if $X$ is Carath\'{e}odory hyperbolic and every closed ball (with respect to the Carath\'{e}odory distance $c_X$) in $X$ is compact.
It is known that a $c$-finitely compact space is Stein (cf. \cite[Corollary 4.4.7]{Kob}).

\begin{lem}\label{lem:cxuniv}
Let $X$ be a $c$-finitely compact space, $\Omega\Subset\C^n$ a bounded convex domain and $\tau\in\Aut X$ be an automorphism such that $\{\tau^j\}_{j\in\N}$ is compactly divergent.
Then there exists an $\cO(X,\Omega)$-universal map for $\tau$.
\end{lem}

\begin{proof}
By Birkhoff's transitivity theorem (cf. \cite[Theorem 1.16]{GroPer}), it suffices to show that for any pair of nonempty open subsets $\cU,\cV\subset\cO(X,\Omega)$ there exists $j\in\N$ such that $(\tau^*)^j(\cU)\cap\cV\neq\emptyset$.
Take $u\in\cU,\ v\in\cV$ with relatively compact images $u(X),v(X)\Subset\Omega$.
Since $\Omega$ is taut, there exists a subsequence $\{\tau^{j_k}\}_{k\in\N}$ of $\{\tau^j\}$ such that $\{u\circ\tau^{j_k}\}_{k\in\N}$ and $\{v\circ\tau^{-j_k}\}_{k\in\N}$ have limits $\widetilde u,\widetilde v\in\cO(X,\Omega)$, respectively.
For $j=1,2$, we denote by $\pi_j$ the projection map $\Omega^2\to\Omega,\ (x_1,x_2)\mapsto x_j$.
Since $\pi_1(u,\widetilde v)=u\in\cU,\ \pi_2(\widetilde u,v)=v\in\cV$,  there exists a holomorphic map $f:\overline\D\to\cO(\Omega^2,\Omega)$ such that $f(1)\circ(u,\widetilde v)\in\cU$ and $f(-1)\circ(\widetilde u,v)\in\cV$ by Theorem \ref{thm:ddom}.

Take $x\in X$.
Since $\{\tau^{j_k}\}_{k\in\N}$ is compactly divergent and $X$ is $c$-finitely compact, $c_X(x,\tau^{j_k}(x))\to\infty$ as $k\to\infty$.
Thus there exists a sequence $\{\varphi_k:\Omega\to\D\}_{k\in\N}$ of holomorphic maps such that $\varphi_k(x)\to -1$ and $\varphi_k\circ\tau^{j_k}(x)\to 1$.
By Montel's theorem and the maximum principle, passing to a subsequence if necessary, we may assume that $\varphi_k\to -1$ and $\varphi_k\circ\tau^{j_k}\to 1$.

For each $k\in\N$, let $F_k=\ev(f\circ\varphi_k,(u,v\circ\tau^{-j_k}))\in\cO(X,\Omega)$.
By our construction, $F_k\to f(1)\circ(u,\widetilde v)\in\cU$ and $F_k\circ\tau^{j_k}\to f(-1)\circ(\widetilde u,v)\in\cV$ as $k\to\infty$.
Thus it follows that $(\tau^*)^{j_k}(\cU)\cap\cV\neq\emptyset$ for any sufficiently large $k\in\N$.
\end{proof}

The following corollary and Theorem \ref{thm:ddom} imply Theorem \ref{thm:univcvx}.

\begin{cor}\label{cor:ddomuniv}
Let $\Omega\Subset\C^n$ be a bounded convex domain, $\tau\in\Aut\Omega$ a generalized translation and $Y$ a complex space.
Let $\cZ\subset\cO(\Omega,Y)$ be a $\tau^*$-invariant irreducible component (with respect to the Zariski topology, see Definition \ref{defn:zar}).
Then $\cZ$ is $\D$-dominated if and only if there exists a $\cZ$-universal map for $\tau$.
\end{cor}

\begin{proof}
Let us assume that $\cZ$ is $\D$-dominated, i.e. there exists a dense holomorphic disc $f:\D\to\cZ$.
Let $\widehat{f}:\D\times\Omega\to Y$ be the associated holomorphic map.
Since $\cO(\Omega,\D\times\Omega)$ is irreducible (it follows from Theorem \ref{thm:ddom}), it defines a dense holomorphic map $\widehat{f}_*:\cO(\Omega,\D\times\Omega)\to\cZ$.
Since the bounded convex domain $\Omega$ is $c$-finitely compact (cf. \cite[Corollary 4.1.10]{Kob}), there exists an $\cO(\Omega,\D\times\Omega)$-universal map $F$ by Lemma \ref{lem:cxuniv}.
Then its image $\widehat{f}_*(F)=F\circ f$ is a $\cZ$-universal map since $\tau^*\circ\widehat{f}_*=\widehat{f}_*\circ\tau^*$.

If there exists a $\cZ$-universal map $F:\Omega\to Y$, then a dense holomorphic disc $f:\D\to\cO(\Omega,\Omega)$ (see Theorem \ref{thm:ddom}) and the postcomposition map $F_*:\cO(\Omega,\Omega)\to\cZ$ give a dense holomorphic disc $F_*\circ f:\D\to\cZ$.
\end{proof}

\begin{rem}\label{rem:her}
From the proof of Lemma \ref{lem:cxuniv}, in the situation of Theorem \ref{thm:univcvx}, it also follows that the precomposition map $\tau^*:\cO(\Omega,Y)\to\cO(\Omega,Y)$ is {\it hereditarily hypercyclic}, i.e. for any subsequence $\{\tau^{j_k}\}_{k\in\N}$ of $\{\tau^j\}_{j\in\N}$ there exists $F\in\cO(\Omega,Y)$ such that $\{F\circ\tau^{j_k}\}_{k\in\N}$ is dense in $\cO(\Omega,Y)$ (see also \cite
[Theorem 14]{Zaj}).
\end{rem}

As an application, we may obtain the following stronger version of Theorem \ref{thm:acconn}.

\begin{cor}
Let $Y$ be a connected complex manifold.
Assume that for any bounded convex domain $\Omega\Subset\C^n$ there exists a nonempty open subset $\cW\subset\cO(\Omega,Y)$ satisfying the following property: for any pair of nonempty open subsets $\cU,\cV\subset W$ there exists an entire curve $f:\C\to\cO(\Omega,Y)$ such that $f(0)\in\cU,\ f(1)\in\cV$.
Then Y is Oka.
\end{cor}

\begin{proof}
We verify the condition in Theorem \ref{thm:acconn}.
Let $\Omega\Subset\C^n$ be a bounded convex domain and $\cU,\cV$ be a pair of nonempty open subsets of $\cO(\Omega,Y)$.
The inclusion map $\iota:\Omega\hookrightarrow\D\times\Omega,\ x\mapsto (0,x)$ induces the surjective precomposition map $\iota^*:\cO(\D\times\Omega,Y)\to\cO(\Omega,Y)$.
We write $\cU'=(\iota^*)^{-1}(\cU)$ and $\cV'=(\iota^*)^{-1}(\cV)$.

Note that $\D\times\Omega$ has a generalized translation $\tau\in\Aut(\D\times\Omega)$.
Let $\cW\subset\cO(\D\times\Omega,Y)$ be an open set as in the assumption.
By the hereditary hypercyclicity of $\tau^*:\cO(\D\times\Omega,Y)\to\cO(\D\times\Omega,Y)$ (see Remark \ref{rem:her}), we may find $j\in\N$ such that $(\tau^*)^j(\cW)\cap\cU'\neq\emptyset$ and $(\tau^*)^j(\cW)\cap\cV'\neq\emptyset$.
By assumption, there exists an entire curve $f:\C\to\cO(\Omega,Y)$ such that $f(0)\circ\tau^j\in\cU'$ and $f(1)\circ\tau^j\in\cV'$.
Then the entire curve $\widetilde f=(\tau^j\circ\iota)^*\circ f:\C\to\cO(\Omega,Y)$ satisfies $\widetilde f(0)\in\cU$ and $\widetilde f(1)\in\cV$.
\end{proof}

From the above corollary, it follows that a connected complex manifold $Y$ is Oka if and only if for any bounded convex domain $\Omega\Subset\C^n$ there exists a somewhere dense entire curve $\C\to\cO(\Omega,Y)$ (compare with Theorem \ref{thm:cdomoka}).

\subsection{Universal maps from general Stein spaces}

We shall show that the $\C$-dominability implies the existence of a universal map from a more general Stein space than bounded convex domains.
Theorem \ref{thm:cdomuniv} and Corollary \ref{cor:cdomuniv} are stated for path components, but the analogous ones hold for irreducible components.
At present, the author does not know whether irreducible components are path-connected.

\begin{thm}\label{thm:cdomuniv}
Let $X$ be a Stein space, $\tau\in\Aut X$ a generalized translation and $Y$ a complex space.
Let $\cP$ be a $\tau^*$-invariant path component of $\cO(X,Y)$ and assume that $\cP$ is $\C$-dominated.
Then there exists a $\cP$-universal map for $\tau$.
\end{thm}

\begin{proof}
Let $f:\C\to\cP$ be a dense entire curve.
Choose a countable base $\{\cU_j\}_{j\in\N}$ for the topology of $\cP$ and a distance function $d$ of $X$.
Assume that $\cU_j\neq\emptyset$ for all $j\in\N$.
For each $j\in\N$, there exist a compact subset $K_j\subset X$, a holomorphic map $h_j\in\cP$ and $\varepsilon_j>0$ such that $\{g\in\cP:\sup_{K_j}d(g,h_j)<\varepsilon_j\}\subset\cU_j$.
For each $j,k\in\N$, we denote by $\cV_{j,k}$ the open set of holomorphic functions $g\in\cO(X)$ satisfying ${\ev}(f\circ g,\id_X)\in(\tau^*)^{-k}(\cU_j)$.
Note that, since $f:\C\to\cP$ has dense image, for each $j,k\in\N$ there exist $a_{j,k}\in\C$ and $\varepsilon_{j,k}>0$ such that $\{g\in\cO(X):\sup_{\tau^k(K_j)}|g-a_{j,k}|<\varepsilon_{j,k}\}\subset\cV_{j,k}$.

Let us show that $\bigcup_{k\in\N}\cV_{j,k}$ is dense in $\cO(X)$ for each $j\in\N$.
Take a holomorphic function $h\in\cO(X)$, a compact $\cO(X)$-convex subset $L\supset K_j$ of $X$ and $\varepsilon>0$.
By definition, there exists $k\in\N$ such that $\tau^k(L)\cap L=\emptyset$ and $\tau^k(L)\cup L$ is $\cO(X)$-convex.
By Oka-Weil theorem, there exists a holomorphic function $g\in\cO(X)$ such that
\begin{enumerate}
\renewcommand{\labelenumi}{(\alph{enumi})}
\item $\sup_L|g-h|<\varepsilon$, and
\item $\sup_{\tau^k(K_j)}|g-a_{j,k}|<\varepsilon_{j,k}$.
\end{enumerate}
This proves the density of $\bigcup_{k\in\N}\cV_{j,k}$.
It follows that $\bigcap_{j\in\N}\bigcup_{k\in\N}\cV_{j,k}\neq\emptyset$ since $\cO(X)$ is a Baire space.
This means that there exists a holomorphic function $\varphi\in\cO(X)$ such that for each $j\in\N$ there exists $k\in\N$ such that ${\ev}(f\circ\varphi\circ\tau^k,\tau^k)\in\cU_j$.
Then $\ev(f\circ\varphi,\id_X)\in\cP$ is a $\cP$-universal map for $\tau$.
\end{proof}

By Theorem \ref{thm:cdomoka} and Theorem \ref{thm:cdomuniv}, we obtain the implication $(1)\implies(2)$ in Theorem \ref{thm:univoka}.
Let us prove the remaining implication $(3)\implies(1)$.

\begin{proof}[Proof of Theorem \ref{thm:univoka}]
Let $\Omega\Subset\C^n$ be a bounded convex domain.
There exists a generalized translation $\tau\in\Aut(\C\times\Omega)$ of the form $\tau=\sigma\times\id_\Omega$ (see Example \ref{ex:gentransl}).
By assumption, there exists a universal map $F:\C\times\Omega\to Y$ for $\tau$.
Note that there exists an entire curve $f:\C\to\cO(\C\times\Omega,\C\times\Omega)$ such that $\{(\tau\times\id_\Omega)^j\}_{j\in\N}\subset\overline{f(\C)}$. 
Using the postcomposition map $F_*:\cO(\C\times\Omega,\C\times\Omega)\to\cO(\C\times\Omega,Y)$, we obtain a dense entire curve $F_*\circ f:\C\to\cO(\C\times\Omega,Y)$.
Note that this also gives a dense entire curve $\C\to\cO(\Omega,Y)$.
Thus $Y$ is Oka by Theorem \ref{thm:cdomoka}.
\end{proof}

A Stein manifold $X$ is called an {\it elliptic} Stein manifold if it is also Oka, because it is equivalent to Gromov's ellipticity (see \cite[Proposition 5.15.2]{For1}).
The following corollary is an analogue of Corollary \ref{cor:ddomuniv} for the $\C$-dominability.

\begin{cor}\label{cor:cdomuniv}
Let $X$ be an elliptic Stein manifold, $\tau\in\Aut X$ a generalized translation and $Y$ a complex space.
Let $\cP$ be a $\tau^*$-invariant path component of $\cO(X,Y)$ and assume that $\tau$ is homotopic to the identity map $\id_X$.
Then $\cP$ is $\C$-dominated if and only if there exists a $\cP$-universal map for $\tau$.
\end{cor}

\begin{proof}
We only need to prove the ``if'' part.
By Theorem \ref{thm:okaprinciple}, $\tau$ lies in the path component $\cP\subset\cO(X,X)$ containing $\id_X$.
Thus $\{\tau^j\}_{j\in\N}\subset\cP$.
By Theorem \ref{thm:cdomoka}, $\cP$ is $\C$-dominated.
The rest of the proof goes through just as in Corollary \ref{cor:ddomuniv}.
\end{proof}

\section*{Acknowledgement}

I would like to thank my supervisor Katsutoshi Yamanoi for his interest and encouragement.

\end{document}